\documentclass[12pt]{article}
\usepackage{amsmath,amsthm}
\usepackage{amssymb,latexsym}
\usepackage{enumerate}
\usepackage{amsfonts}
\usepackage{graphicx}
\usepackage[latin1]{inputenc}

\newtheorem{thm}{Theorem}[section]
\newtheorem{cor}[thm]{Corollary}
\newtheorem{lem}[thm]{Lemma}

\theoremstyle{definition}
\newtheorem{defin}[thm]{Definition}
\newtheorem{rem}[thm]{Remark}
\newtheorem{exa}[thm]{Example}
\newtheorem{prop}[thm]{Proposition}

\numberwithin{equation}{section}

\def\be   {\begin{equation}}   \def\ee   {\end{equation}}
\def\ba   {\begin{array}}      \def\ea   {\end{array}}
\def\bea  {\begin{eqnarray}}   \def\eea  {\end{eqnarray}}
\def\bean {\begin{eqnarray*}}  \def\eean {\end{eqnarray*}}

\newcommand{\pre}{\mathrm{Re}}

\newcommand{\bi} {\ensuremath{{\bf i}}}
\newcommand{\bo} {\ensuremath{{\bf i_1}}}

\newcommand{\bt} {\ensuremath{{\bf i_2}}}

\newcommand {\bj}{\ensuremath{{\bf j}}}

\newcommand{\eo} {\ensuremath{{\bf e_1}}}
\newcommand{\et} {\ensuremath{{\bf e_2}}}

\newcommand{\mC}{\ensuremath{\mathbb{C}}}

\newcommand{\mR}{\ensuremath{\mathbb{R}}}

\begin{document}


\baselineskip=17pt


\title{Normal Families of Bicomplex Meromorphic Functions}

\author{Kuldeep Singh Charak\\
Department of Mathematics, University of Jammu\\
Jammu-180 006, India\\
E-mail: kscharak7@rediffmail.com
\and
Dominic Rochon\\
D\'epartement de math\'ematiques et d'informatique\\
Universit\'e du Qu\'ebec \`a Trois-Rivi\`eres\\
C.P. 500 Trois-Rivi\`eres, Qu\'ebec, Canada G9A 5H7\\
E-mail: Dominic.Rochon@UQTR.CA
\and
Narinder Sharma\\
Department of Mathematics, University of Jammu\\
Jammu-180 006, India\\
E-mail: narinder25sharma@sify.com}

\date{}

\maketitle


\renewcommand{\thefootnote}{}
\footnote{2010 \emph{Mathematics Subject Classification}: 30G, 30G35, 30D30, 32A, 32A30.}

\footnote{\emph{Key words and phrases}: Bicomplex numbers, Complex Clifford algebras, Normal families,
Bicomplex holomorphic functions, Bicomplex meromorphic functions, Bicomplex Riemann sphere.}

\renewcommand{\thefootnote}{\arabic{footnote}}
\setcounter{footnote}{0}


\begin{abstract}
In the present paper, we introduced the extended bicomplex plane $\overline{\mathbb{T}}$,
its geometric model: the bicomplex Riemann sphere, and the bicomplex chordal metric that enables
us to talk about the convergence of the sequences of bicomplex meromorphic functions. Hence
the concept of the normality of a family of bicomplex meromorphic functions on bicomplex domains emerges.
Besides obtaining a normality criterion for such families, the bicomplex analog of the Montel theorem for
meromorphic functions  and the Fundamental Normality Tests for families of bicomplex holomorphic functions
and bicomplex meromorphic functions are also obtained.
\end{abstract}

\section{Introduction}
The concept of normality of a family of bicomplex holomorphic
functions was introduced in \cite{Sharma}, and we now intend to study
the same property for a family of bicomplex meromorphic functions.
A family $\boldsymbol F$ of meromorphic functions on a domain
$D\subset\mathbb{C}$ is said to be normal in $D$ if every sequence
in $\boldsymbol F$ contains a subsequence which converges
uniformly on compact subsets of $D$; the limit function is either
meromorphic in $D$ or identically equal to $\infty$. Of course,
the convergence in this situation is with respect to the chordal
metric on the Riemann sphere $\overline {\mathbb C} =\mathbb C
\cup \{\infty\}$ (cf. \cite{Schiff}). Unfortunately, the one complex variable
case doesn't admit any simple generalization to extend facts to the bicomplex case.

In order to discuss the convergence of sequences of bicomplex
meromorphic functions on bicomplex plane domains, we introduce the
extended bicomplex plane $\overline{\mathbb{T}}$, its
geometric model viz., the bicomplex Riemann sphere,  the
bicomplex chordal metric on the bicomplex Riemann sphere,
and the idea of convergence on
$\overline{\mathbb{T}}$. In turn, these developments facilitate
the introduction of the concept of the normality of a family of
bicomplex meromorphic functions on bicomplex domains. This forms
the content of Section~3.

In Section~4 of the paper, after introducing the concept of
normality of a family of bicomplex meromorphic functions, a
normality criterion for such families, the bicomplex analog of the
Montel theorem for meromorphic functions and the Fundamental
Normality Tests for families of bicomplex holomorphic functions
and bicomplex meromorphic functions are also obtained.

\section{Preliminaries}
As in \cite{Shapiro} (see also \cite{Charak} and \cite{Sharma}), the algebra of bicomplex numbers \begin{equation} \mathbb{T}:=\{z_1+z_2\bt\ |\ z_1, z_2 \in
\mathbb{C}(\bo) \} \label{enstetra} \end{equation} is the space isomorphic to $\mathbb{R}^{4}$ via the map
$$z_1+z_2\bt =x_0+x_1\bo+x_2\bt+x_3\bj \rightarrow (x_0,x_1,x_2,x_3)\in\mathbb{R}^{4},$$
and the multiplication is defined using the following rules:
$$\bo^2=\bt^2=-1,\qquad\bo\bt=\bt\bo=\bj\quad\mbox{ so that }\quad\bj^2=1.$$
Note that we define $\mC(\bi_k):=\{x+y\bi_k\ |\ \bi_k^2= -1$
and $x,y\in \mR \}$ for $k=1,2$. Hence, it is easy to see that the multiplication of two bicomplex numbers is commutative. In fact, the bicomplex numbers $$\mathbb{T}\cong {\rm Cl}_{\Bbb{C}}(1,0) \cong {\rm Cl}_{\Bbb{C}}(0,1)$$
are \textbf{unique} among the \textbf{complex Clifford algebras} (see \cite{Brackx, Delanghe} and \cite{Ryan})
in that they are commutative but not division algebra.
Also, since the map $z_1+z_2\bt \rightarrow (z_1,z_2)$ gives a natural isomorphism between the $\mathbb{C}$-vector spaces $\mathbb{T}$ and
$\mathbb{C}^{2}$, we have $\mathbb{T}=\mathbb{C}\otimes_{\mathbb{R}}\mathbb{C}$. That is, we can view the algebra $\mathbb{T}$ as the complexified $\mathbb{C}(\bo)$ exactly the way $\mathbb{C}$ is complexified $\mathbb{R}$.
In particular, in the equation (\ref{enstetra}), if we put $z_1=x$ and
$z_2=y\bo$ with $x,y \in \mathbb{R}$, then we obtain the following
subalgebra of hyperbolic numbers, also called duplex numbers (see, e.g. \cite{Shapiro}, \cite{Sobczyk}):
$$\mathbb{D}:=\{x+y\bj\ |\ \bj^2=1,\ x,y\in \mathbb{R}\}\cong {\rm Cl}_{\Bbb{R}}(0,1).$$

\noindent The two projection maps $\mathcal{P}_1,\mathcal{P}_2:\mathbb{T}\longrightarrow\mathbb{C}(\bold{i_1})$ defined by
\begin{equation}\label{projections}
\mathcal{P}_1(z_1+z_2\bold{i_2})=z_1-z_2\bold{i_1}\qquad\mbox{ and }\qquad\mathcal{P}_2(z_1+z_2\bold{i_2})=z_1+z_2\bold{i_1},
\end{equation}
are used extensively in the sequel.

The complex (square) norm $CN(w)$ of the bicomplex number
$w$ is the complex number ${z_1}^2 +{z_2}^2$; writing
$w^*=z_1-z_2\bt$, we see that $CN(w)=ww^*$. Then a bicomplex
number $w=z_1+z_2\bt$ is invertible if and only if $CN(w)\neq 0$.
Precisely, \begin{equation*} w^{-1}= \displaystyle \frac{w^*}{CN(w)}. \end{equation*} The
set of units in the algebra $\mathbb{T}$ forms a multiplicative
group which we shall denote by $\mathbb{T}_*$ (see \cite{Baird}). Unlike the algebra
$\mathbb{C}$, the bicomplex algebra $\mathbb{T}$  has zero
divisors given by \begin{equation*} \mathcal{NC}=\{w \in \mathbb{T}:
CN(w)=0\}=\{z(1\pm\bj)|\ z\in \mathbb{C}(\bo)\}, \end{equation*}
which we may call the {\em null-cone}. Note that, using orthogonal
idempotents $$\bold{e_1}=\frac{1+\bold{j}}{2},\qquad
\bold{e_2}=\frac{1-\bold{j}}{2},\ \text{ in }\ \mathcal{NC},$$
each bicomplex number  $w=z_1+z_2\bt \in \mathbb{T}$ can be
expressed uniquely as \begin{equation*}
w=\mathcal{P}_1(w)\eo+\mathcal{P}_2(w)\et, \end{equation*} where
$\mathcal{P}_1$ and $\mathcal{P}_2$ are projection maps defined in
$(\ref{projections})$. This representation of $\mathbb{T}$ as
$\mathbb{C}\oplus\mathbb{C}$ helps to do addition, multiplication
and division term-by-term. With this representation we can
directly express $|w|_j$ as
$$|w|_j := |\mathcal{P}_1(w)|\eo+|\mathcal{P}_2(w)|\et$$
and will be referred to as the
$\bold{j}$\textsl{\textbf{-modulus}} of
$w=z_1+z_2\bt\in\mathbb{T}$ (see \cite{Shapiro}).
\begin{defin} Let $X_1$ and $X_2$ be subsets of $\mathbb{C}(\bold{i_1})$. Then the following set
$$X_{1}\times_e X_{2}:=\{w=z_1+z_2\bold{i_2}\in\mathbb{T}~:~\mathcal{P}_1(w)\in X_1\ \text{ and }\ \mathcal{P}_2(w)\in X_2\}$$
is called a $\mathbb{T}$\textsl{\textbf{-cartesian set}} determined by $X_1$ and $X_2$, where $\mathcal{P}_1$ and $\mathcal{P}_2$ are projections as defined in \eqref{projections}.
\end{defin}
\noindent It is easy to see that if $X_1$ and $X_2$ are domains (open and connected) of $\mathbb{C}(\bold{i_1})$ then $X_1\times_e X_2$ is also a domain of $\mathbb{T}$. We define the ``discus" with center $a=a_1+a_2\bt$ of radius $r_1$ and $r_2$ of $\mathbb{T}$ as follows \cite{Price}:
\small
\begin{eqnarray*}
D(a;r_1,r_2)&=& B^{1}(a_1-a_2\bo,r_1)\times_{e}B^{1}(a_1+a_2\bo,r_2)\\
&=& \{w_1\bold{e_1}+w_2\bold{e_2}: |w_1-(a_1-a_2\bo)|<r_1,|w_2-(a_1+a_2\bo)|<r_2\},
\end{eqnarray*}
\normalsize
where $B^{n}(z,r)$ is an open ball with center $z\in\mathbb{C}^n(\bold{i_1})$ and radius $r>0$. In the particular case where $r=r_1=r_2$, $D(a;r,r)$ will be called the $\mathbb{T}$-disc with center $a$ and radius $r$. In particular, we define $$\overline{D}(a;r_1,r_2):=\overline{B^{1}(a_1-a_2\bo,r_1)}\times_{e}\overline{B^{1}(a_1+a_2\bo,r_2)}\subset \overline{D(a;r_1,r_2)}.$$ We remark that $D(0;r,r)$ is, in fact, the \textbf{Lie Ball} (see \cite{Avanissian}) of radius $r$ in $\mathbb{T}$.

Further, the projections as defined in \eqref{projections}, help to understand bicomplex holomorphic functions in
terms of the following Ringleb's Decomposition Lemma \cite{Riley}.

\begin{thm}\label{theo5}
Let $\Omega\subset\mathbb{T}$ be an open set. A function $f:\Omega\longrightarrow\mathbb{T}$ is $\mathbb{T}$-holomorphic
on $\Omega$ if and only if the two natural functions $f_{e1}:\mathcal{P}_1(\Omega)\longrightarrow\mathbb{C}(\bold{i_1})$ and
$f_{e2}:\mathcal{P}_2(\Omega)\longrightarrow\mathbb{C}(\bold{i_1})$ are holomorphic, and
$$f(w)=f_{e1}(\mathcal{P}_1(w))\eo+f_{e2}(\mathcal{P}_2(w))\et,\ \forall\mbox{ }w=z_1+z_2\bold{i_2}\in \Omega,$$
\end{thm}

The Ringleb's Lemma for bicomplex meromorphic functions is as follows \cite{Charak}.
\begin{thm}\label{mero1} Let $\Omega\subset\mathbb{T}$ be an open set. A function $f:\Omega\longrightarrow\mathbb{T}$ is bicomplex meromorphic on $\Omega$ if and only if the two natural functions $f_{e1}:\mathcal{P}_1(\Omega)\longrightarrow\mathbb{C}(\bold{i_1})$ and $f_{e2}:\mathcal{P}_2(\Omega)\longrightarrow\mathbb{C}(\bold{i_1})$ are meromorphic, and
$$f(w)=f_{e1}(\mathcal{P}_1(w))\eo+f_{e2}(\mathcal{P}_2(w))\et,\ \forall\mbox{ }w=z_1+z_2\bold{i_2}\in \Omega.$$
\end{thm}
\begin{defin} Let $f:\Omega\longrightarrow\mathbb{T}$ be a bicomplex meromorphic
function on the open set $\Omega\subset\mathbb{T}$, and let $f_{e1}:\mathcal{P}_1(\Omega)\longrightarrow\mathbb{C}(\bo)$ and $f_{e2}:\mathcal{P}_2(\Omega)\longrightarrow\mathbb{C}(\bo)$ be the natural maps. Then we say that $w=\mathcal{P}_1(w)\eo+\mathcal{P}_2(w)\et\in \Omega$ is a (strong) \textsl{\textbf{pole}} for the bicomplex meromorphic function
$$f(w)=f_{e1}\mathcal{P}_1(w)\eo+f_{e1}\mathcal{P}_2(w)\et$$
if $\mathcal{P}_1(w)$ (and) or $\mathcal{P}_2(w)$ is a pole for
$f_{e1}$ or $f_{e2}$, respectively. \label{pole}
\end{defin}

\begin{rem}
Poles of bicomplex meromorphic functions are not isolated singularities.
\end{rem}

\noindent It is also easy to obtain the following characterization
of poles.

\begin{prop}
Let $f:X\longrightarrow\mathbb{T}$ be a bicomplex meromorphic
function on the open set $\Omega\subset\mathbb{T}$. If $w_0\in \Omega$ then $w_0$ is a pole of $f$ if and only if
$$\mathop {\lim }\limits_{w \to w_0 } \left| {f(w)} \right| = \infty .$$
\end{prop}

\noindent A classical example of bicomplex meromorphic function is the \textbf{bicomplex Riemann zeta} function introduced by Rochon in \cite{Rochon1}.

\section {The Extended Bicomplex Plane $\overline{\mathbb{T}}$}

Since the range of bicomplex meromorphic function lies beyond the
bicomplex plane, we need the {\bf extended bicomplex plane} to
study the bicomplex meromorphic functions. Further, it would help
to study the limit points of unbounded sets in bicomplex plane. We
obtain this extended bicomplex plane by using extended $\mathbb
\mC(\bo)$-plane.

For, we may consider the set
\small
\begin{eqnarray*}
\overline{\mC(\bo)} \times_{e} \overline{\mC(\bo)} &=& \left(\mC(\bo) \cup \{ \infty \} \right)\times_{e}\left( \mC(\bo)\cup \{ \infty \}\right)\\
&=& \left(\mC(\bo) \times_{e} \mC(\bo)\right) \cup \left(\mC(\bo) \times_{e} \{ \infty \} \right)\cup \left(\{ \infty \} \times_{e} \mC(\bo)\right) \cup \{ \infty \}\\
&=& {\mathbb{T}} \cup {I_{\infty}},
\end{eqnarray*}
\normalsize
writing $I_{\infty}$ for the set $\left({\mC(\bo) \times_{e} \{ \infty \}}\right) \cup \left({\{ \infty \} \times_{e} \mC(\bo)}\right) \cup {\{ \infty \}}$. Clearly, any unbounded sequence in $\mathbb{T}$ will have a limit point in $I_{\infty}$.
\begin{defin} The set $\overline{\mathbb{T}}=\overline{\mC(\bo)} \times_{e} \overline{\mC(\bo)}$ is called the \textbf{extended bicomplex plane}. That is, $$\overline{\mathbb{T}} = \mathbb{T} \cup I_{\infty},\quad\mbox{ with }\quad I_{\infty} = \{w \in \overline{\mathbb{T}} : \left\|w\right\| = \infty \}.$$
\end{defin}

It is of significant importance to observe that formation of the extended bicomplex plane $\overline{\mathbb{T}}$ requires us to add an infinity set viz. $I_{\infty}$, which we may call the \textbf{bicomplex infinity set}.

We need some definitions in order to give a characterization of
this set.
\begin{defin}An element $w \in I_{\infty}$ is said to be a $\mathcal{P}_1$-infinity ($\mathcal{P}_2$-infinity) element if
$\mathcal{P}_{1}(w)= \infty \ (\mathcal{P}_{2}(w)= \infty)$ and $\mathcal{P}_{2}(w)\neq \infty \ (\mathcal{P}_{1}(w)\neq \infty).$
\end{defin}

\begin{defin}The set of all $\mathcal{P}_1$-infinity elements is called the $I_1$\textsl{\textbf{-infinity set}}. It is denoted by $I_{1,\infty}$. Therefore,
$$I_{1,\infty} = \{w \in \overline{\mathbb{T}} : \mathcal{P}_{1}(w)= \infty ,\\\ \mathcal{P}_{2}(w)\neq \infty \}.$$ Similarly we can define the $I_2$\textsl{\textbf{-infinity set}} as:
$$I_{2,\infty} = \{w \in \overline{\mathbb{T}} : \mathcal{P}_{1}(w) \neq \infty ,\\\ \mathcal{P}_{2}(w)= \infty \}.$$
\end{defin}

\begin{defin}
An element $w \in \overline{\mathbb{T}}$ is said to be a $\mathcal{P}_1$-zero ($\mathcal{P}_2$-zero)
element if $\mathcal{P}_{1}(w)= 0 \  (\mathcal{P}_{2}(w)= 0)$ and $\mathcal{P}_{2}(w)\neq 0 \ (\mathcal{P}_{1}(w)\neq 0).$
\end{defin}

\begin{defin} The set of all $\mathcal{P}_1$-zero elements is called the $I_1$\textsl{\textbf{-zero set}}; it is denoted by $I_{1,0}$. That is, $I_{1,0} = \{w \in \overline{\mathbb{T}}: \mathcal{P}_{1}(w)= 0 ,\mathcal{P}_{2}(w)\neq 0 \}$. Similarly, we may define the $I_2$\textsl{\textbf{-zero set}} as the set $\{w \in \overline{\mathbb{T}}: \mathcal{P}_{1}(w) \neq 0 ,\mathcal{P}_{2}(w)= 0 \}$.
\end{defin}

We now construct the following two new sets:
\begin{equation*}
I^{-}_{\infty} = I_{1,\infty} \cup I_{2,\infty},\qquad I^{-}_{0} = I_{1,0} \cup I_{2,0},
\end{equation*}
so that $I_{\infty} = I^{-}_{\infty} \cup \{ \infty \}$ and $\mathcal{NC} = I^{-}_{0} \cup \{0\}$.
With these definitions, each element in the null-cone has an
inverse in $I_{\infty}$ and vice versa. One can easily check that
the elements of the set $I^{-}_{\infty}$ do not satisfy all the
properties as satisfied by the $\mC(\bo)$-infinity but the element
$\infty = \infty \eo + \infty \et $ does. We may call the set
$I^{-}_{\infty}$, the \textbf{weak bicomplex infinity set} and the
element $ \infty = \infty \eo + \infty \et $, the \textbf{strong
infinity}. This nature of the set $ I_{\infty}$ generates the idea
of weak and strong poles for bicomplex meromorphic functions (see
\cite{Charak}). Now, in order to work in the extended bicomplex plane,
it is desirable to have a geometric model wherein the elements of
$\overline{\mathbb{T}}$ have a concrete representative so as to
treat the points of $I_{\infty}$ as good as any other point of
$\overline{\mathbb{T}}$. To obtain such a model, one can use the
usual stereographic projections of $\overline{\mC(\bo)}$ as two
components in the  idempotent decomposition to get a one-to-one
and onto correspondence between the points of $S \times S$, where
$S$ is the unit sphere in $\mathbb{R}^{3}$, and
$\overline{\mathbb{T}}$. Hence, we can visualize the extended
bicomplex plane directly in $\mathbb{R}^{6}=\mathbb{R}^{3} \times
\mathbb {R}^{3}$. With this representation, we call
$\overline{\mathbb{T}} $ the \textbf{bicomplex Riemann sphere}.

Observe that what is done above is basically a compactification of
$\mathbb{C}^2$, using bicomplex setting. That is, suitable points
at infinity are added to $\mathbb{T}$ to get the extended
bicomplex plane $\overline{\mathbb{T}}$. In higher dimensions such
compactifications are well known under the name, conformal
compactifications. In fact, such compactifications are obtained as
homogeneous spaces of Lie groups (see \cite{Baird} and \cite{Baston}).

\subsection{The Chordal Metric on $\overline{\mathbb{T}}$}
To initiate a study of normal families of bicomplex meromorphic functions, we first have to extend the chordal distance to the extended bicomplex plane in such a way that facilitates introduction of notions like convergence of sequences and continuity of bicomplex meromorphic functions. The chordal metric on $\overline{\mC(\bo)}$ can be used to define a distance on $\overline{\mathbb T}.$

\begin{prop}
If $\chi : \overline{\mC(\bo)} \times \overline{\mC(\bo)}\longrightarrow \mathbb{R}$ be the chordal metric on
$\overline{\mC(\bo)}$. Then the mapping $ \mathbb \chi_e  :  \overline{\mathbb{T}}\times \overline{\mathbb{T}}\longrightarrow \mathbb R \ \ \text{defined as:}$
\begin{equation*}
\mathbb \chi_e\left({z,w}\right) =\sqrt{\frac{\chi^{2}(\mathcal{P}_{1}(z),\mathcal{P}_{1}(w)) + \chi^{2}(\mathcal{P}_{2}(z),\mathcal{P}_{2}(w))}{2}}
\end{equation*}
is a metric on $\overline{\mathbb{T}}.$
\end{prop}
\begin{proof}
It is easy to verify that $\forall z,w\in\overline{\mathbb{T}}$ we have:
\begin{equation*}
\mathbb \chi_e\left({z,w}\right)\geq 0;
\end{equation*}
\begin{equation*}
\mathbb \chi_e\left({z,w}\right) = 0  \ \ \text{iff } z = w;
\end{equation*}
\begin{equation*}
\mathbb \chi_e\left({z,w}\right)= \mathbb \chi_e\left({w,z}\right).
\end{equation*}
Now, we show that $\mathbb \chi$ also satisfies the triangle inequality. Let $ z, \ w \ ,\ v \in \overline{\mathbb{T}} .$  \ We have to show that
$$\mathbb \chi_e\left({z,w}\right)\leq \mathbb \chi_e\left({z,v}\right)+ \mathbb \chi_e\left({v,w}\right).$$\\ For this,
$\mathbb \chi_e\left({z,w}\right) = \sqrt{\displaystyle\frac{\chi^{2}\left({\mathcal{P}_{1}(z),\mathcal{P}_{1}(w)}\right) + \chi^{2}\left({\mathcal{P}_{2}(z),\mathcal{P}_{2}(w)}\right)}{2}}$\\

\noindent
$\leq\sqrt{\frac { \left\{{\chi\left({\mathcal{P}_{1}(z),\mathcal{P}_{1}(v)}\right) + \chi\left({\mathcal{P}_{1}(v),\mathcal{P}_{1}(w)}\right)}\right\} ^{2} + \left\{{\chi\left({\mathcal{P}_{2}(z),\mathcal{P}_{2}(v)}\right) + \chi\left({\mathcal{P}_{2}(v),\mathcal{P}_{2}(w)}\right) }\right\}^{2} }{2}}.$

\noindent Now, using Minkowski's inequality in the above inequality, we obtain that
$$ \leq \sqrt{\frac {\chi^{2}\left({\mathcal{P}_{1}(z),\mathcal{P}_{1}(v)}\right) + \chi^{2}\left({\mathcal{P}_{2}(z),\mathcal{P}_{2}(v)}\right)}{2}}$$
$$ + \sqrt{\frac {\chi^{2}\left({\mathcal{P}_{1}(v),\mathcal{P}_{1}(w)}\right) + \chi^{2}\left({\mathcal{P}_{2}(v),\mathcal{P}_{2}(w)}\right)}{2}}$$
$$ = \mathbb \chi_e\left({z,v}\right)+ \mathbb \chi_e\left({v,w}\right).$$
Hence,
$\mathbb \chi_e$
is a metric on $\overline{\mathbb{T}}$.
\end{proof}

\medskip
We call this metric $\mathbb \chi_e$ on $\overline{\mathbb{T}}$ the \textbf{bicomplex chordal metric}. The virtue of the
bicomplex chordal metric is that it allows $w\in I_{\infty}$ to be treated like any other point. Hence, we are able now to analyse the
behavior of the bicomplex meromorphic functions in the extended bicomplex plane, especially on the set $I_{\infty}$.
\begin{rem} As for the $\bj -modulus$, let us define
\begin{equation*}
{\chi_{\bj}}(z, w):=\chi(\mathcal{P}_1(z),\mathcal{P}_1(w))\eo +\chi(\mathcal{P}_2(z), \mathcal{P}_2(w))\et
\end{equation*}
in the extended hyperbolic numbers. Then
$$\pre({\chi_{\bj}}^{2}(z, w))={\chi_{e}}^2(z, w)$$
and thus we have
\begin{equation*}
\chi_{e}(z, w)=\sqrt{\pre({\chi_{\bj}}^{2}(z, w))}
\end{equation*}
where
\begin{equation*}
\mathbb \chi_{\bj}(z,w)=\frac{\left|z-w\right|_{\bj}}{\sqrt{1+\left|z\right|_{\bj}}\sqrt{1+\left|w\right|_{\bj}}}\mbox{ if } z,w\in\mathbb{T}.
\end{equation*}
\end{rem}
Some of the important properties of the bicomplex chordal metric are discussed in the following results.
\begin{thm} If  $ z = z_{1}\eo + z_{2}\et \  \text{and} \ w = w_{1}\eo + w_{2}\et
$ are any two elements in the extended bicomplex plane and  $\mathbb \chi_e$ is the bicomplex chordal metric on $\overline{\mathbb{T}}$. Then,

\bigskip
   $ 1.\ \mathbb \chi_e(z,w)  \leq 1 $;
\bigskip

   $ 2.\ \mathbb \chi_e(0,\infty) =  1 $;
\bigskip

   $ 3.\ \mathbb \chi_e(z,w)= \frac{1}{\sqrt{2}}{\chi(z_{1},\infty)}\  \text{if } \mathcal{P}_{2}(z)=\mathcal{P}_{2}(w)= 0 \ \text{and } \mathcal{P}_{1}(w)= \infty$;
\bigskip

   $ 4.\ \mathbb \chi_e(z,w)= \frac{1}{\sqrt{2}}\chi(z_{1},w_{1}) \ \text{if } \mathcal{P}_{2}(z)=\mathcal{P}_{2}(w)= \infty$;
\bigskip

   $ 5.\ \mathbb \chi_e(z,\infty)= \frac{1}{\sqrt{2}}\chi(z_{2},\infty)\ \text{if } \mathcal{P}_{1}(z)=\infty$;
\bigskip

   $ 6.\  \mathbb \chi_e(z,w) =  \mathbb \chi_e(z^{-1},w^{-1})$;
\bigskip

   $ 7.\ \mathbb \chi_e(z,w)=\chi(z,w)$ if $z,w\in\overline{\mC(\bo)}$;
\bigskip

   $ 8.\ \mathbb \chi_e(z,w) \leq \left\|z-w\right\|$ if $z,w\in\mathbb{T}$;
\bigskip

   $9.\ \mathbb \chi_e(z,w)$ is a continuous function on $\mathbb{T}$.

\end{thm}
The following implication
\begin{equation*}
\left\|z\right\|\leq\left\|w\right\|\Longrightarrow\mathbb\chi(0,z)\leq\mathbb\chi(0,w),\ z,w\in\overline{\mC(\bo)}
\end{equation*}
need not be true in case of the bicomplex chordal metric $\mathbb \chi_e$ on $\overline{\mathbb{T}}$. To support our argument we give the following examples.

\begin{exa} Let $$ z = (1+2\bold{i_{1}})\eo+(2+3\bold{i_{1}})\et \ \text{and} \ \
 w = (1+\bold{i_{1}})\eo+(3+3\bold{i_{1}})\et.$$
Then, $\left\|z\right\| \leq   \left\|w\right\|$, but $\mathbb\chi_e(0,z) = \sqrt{0.88}$ and $\mathbb\chi_e(0,w) = \sqrt{0.80}$ implies that $\chi_e(0,z) > \mathbb\chi_e(0,w)$.
\end{exa}
\begin{exa}Let $$ z = (4+\bold{i_{1}})\eo+(2+3\bold{i_{1}})\et \ \text{and } \\
w = (1+2\bold{i_{1}})\eo+(3+4\bold{i_{1}})\et.$$
Then, $\left\|z\right\| =  \left\|w\right\|$, but $\mathbb\chi_e(0,z) = \sqrt{0.93}$ and $\mathbb\chi_e(0,w) = \sqrt{0.89}$ implies that $\chi_e(0,z) > \mathbb\chi_e(0,w)$.
\end{exa}
However, we can prove the following result.

\begin{prop}Let $z,w\in\mathbb{T}$. If $\left\|z\right\| \leq   \left\|w\right\|$ then
\begin{equation*}
\chi_e(0,z) \leq  \chi_e(0,\sqrt{2}\left\|w\right\|).
\end{equation*}
\end{prop}
\begin{proof} By definition,
$$ \mathbb \chi_e(0,z) =   \sqrt { \frac{\chi^{2}(0,\mathcal{P}_{1}(z)) + \chi^{2}(0,\mathcal{P}_{2}(z))}{2}}$$
\begin{equation*}
= \sqrt{ \frac{1}{2}\left\{ \frac{\left|\mathcal{P}_{1}(z)\right|^2}{1+\left|\mathcal{P}_{1}(z)\right|^2} +
\frac{\left|\mathcal{P}_{2}(z)\right|^2}{1+\left|\mathcal{P}_{2}(z)\right|^2}\right\}}.
\end{equation*}
Since,
$$\left|\mathcal{P}_{i}(z)\right|\leq \sqrt{2}\left\|z\right\|\leq \sqrt{2}\left\|w\right\|\mbox{ for }i=1,2$$
then
$$\chi(0,\mathcal{P}_{i}(z))=\chi(0,\left|\mathcal{P}_{i}(z)\right|) \leq  \chi(0,\sqrt{2}\left\|w\right\|)\mbox{ for }i=1,2.$$
Hence,
$$\chi_e(0,z) \leq \chi_e(0,\sqrt{2}\left\|w\right\|).$$
\end{proof}
\subsection{Convergence in $\overline{\mathbb{T}}$}
\begin{defin}
A sequence of functions $\{f_n\}$ converges \textbf{bispherically uniformly} to a function  $f$ on a set $E\subset\mathbb{T}$ if, for any
$\epsilon>0$, there is a number $n_0$ such that $n\geq n_0$ implies
$$\chi_e(f_n(w),f(w))<\epsilon,$$
for all $w\in E$.
\end{defin}

Note that if $\{f_n\}$ converges uniformly to $f$ on $E\subset\mathbb{T}$, then it also converges spherically uniformly to $f$
on $E$. The converse holds if the limit function is bounded.
\begin{lem} $ \mathbb \chi_e(z,w) \geq\displaystyle\frac{\left\|z-w\right\|}{\sqrt{1+2\left\|z\right\|^2}\sqrt{1+2\left\|w\right\|^2}}$, if $z,w\in\mathbb{T}.$
\label{lemmaineq}
\end{lem}
\begin{proof} We shall establish the validity of the inequality in
this lemma by obtaining an equivalent inequality that holds
trivially.
For $z,w\in\mathbb{T},$ put $\mathcal{P}_1(z)=a, \ \mathcal{P}_2(z)=b, \ \mathcal{P}_1(w)=c, \text{ and } \mathcal{P}_2(w)=d.$ Then \\

~~~~$\mathbb \chi_e(z,w) \geq\displaystyle\frac{\left\|z-w\right\|}{\sqrt{1+2\left\|z\right\|^2}\sqrt{1+2\left\|w\right\|^2}}$\\

$\Leftrightarrow \mathbb \chi_e^2(z,w)
 \geq\displaystyle\frac{{\left\|z-w\right\|}^2}{(1+2\left\|z\right\|^2)(1+2\left\|w\right\|^2)}$\\

$\Leftrightarrow \chi^2(a,c)+ \chi^2(b,d)
 \geq\displaystyle\frac{\left|a- c\right|^2 + \left|b- d\right|^2}{(1 + \left|a\right|^2 + \left|b\right|^2 )(1 + \left|c\right|^2 + \left|d\right|^2)}$ \\

$\Leftrightarrow\displaystyle\frac{\left|a- c\right|^2}{(1 + \left|a\right|^2)(1 + \left|c\right|^2)}
 + \frac{\left|b- d\right|^2}{(1 + \left|b\right|^2)(1 + \left|d\right|^2)}$\\

$\geq\displaystyle\frac{\left|a- c\right|^2 }{(1 + \left|a\right|^2 + \left|b\right|^2 )(1 + \left|c\right|^2 + \left|d\right|^2)}
+ \frac{\left|b-d\right|^2}{(1 + \left|a\right|^2 +  \left|b\right|^2 )(1 + \left|c\right|^2 + \left|d\right|^2)}$\\

$\Leftrightarrow \left|a- c\right|^2 \left[\displaystyle\frac{1}{(1 + \left|a\right|^2)(1 + \left|c\right|^2)} - \frac{1}{(1 + \left|a\right|^2 + \left|b\right|^2 )(1 + \left|c\right|^2 + \left|d\right|^2)}\right]$\\

$\geq  \left|b- d\right|^2 \left[\displaystyle\frac{1}{(1 + \left|a\right|^2 + \left|b\right|^2 )(1 + \left|c\right|^2 + \left|d\right|^2)} -  \frac{1}{(1 + \left|b\right|^2)(1 + \left|d\right|^2)}\right]$\\

$\Leftrightarrow\displaystyle\frac{\left|a- c\right|^2 \left[ \left|d\right|^2 +  \left|b\right|^2 +  \left|a\right|^2 \left|d\right|^2 + \left|b\right|^2 \left|c\right|^2  +  \left|b\right|^2 \left|d\right|^2\right]}{(1 + \left|a\right|^2)(1 + \left|c\right|^2)}$\\

$\geq\displaystyle\frac{\left|b- d\right|^2 \left[-\left\{ \left|c\right|^2 +  \left|a\right|^2 + \left|a\right|^2 \left|c\right|^2  +  \left|a\right|^2 \left|d\right|^2 +  \left|b\right|^2 \left|c\right|^2\right\}\right]}{(1 + \left|b\right|^{2})(1 + \left|d\right|^2)}.  $ \\

The left hand side of the last inequality is a positive real
number where the right hand side is a negative real number and
hence holds trivially. \end{proof}

\begin{thm}
If the sequence $\{f_n\}$ converges bispherically uniformly to a bounded function $f$ on $E\subset \mathbb{T}$,
then $\{f_n\}$ converges uniformly to $f$ on $E$.
\end{thm}
\begin{proof}
Since $\{f_n\}$ converges bispherically uniformly to a bounded function $f$ on $E\subset \mathbb{T}$, for every $\epsilon >0$
there is $n_0$ such that for all $n\geq n_0$, we have
$$\chi_e(f_n(w),f(w))<\epsilon.$$
Now from this inequality and by the definition of the bicomplex
chordal metric it follows that $\mathcal{P}_i(f_n(w))$ converges
uniformly to $\mathcal{P}_i(f(w))$ on $\mathcal{P}_i(E)$, $i=1,2.$
Further, since $f$ is bounded on $E$, $\mathcal{P}_i(f(w))$ is
bounded on $\mathcal{P}_i(E)$, $i=1,2,$ and hence
$\mathcal{P}_i(f_n(w))$ is bounded on $\mathcal{P}_i(E)$, $i=1,2,$
for all but finitely many $n$. This implies that there is a
positive constant $L$ such that
$$\left\|f_n(w)\right\|<L \ \forall n\geq n_0,$$
on $\mathcal{P}_1(E)\times_e \mathcal{P}_2(E) \supseteq E.$
Now by Lemma \ref{lemmaineq}, we have
$$\left\|f_n(w)-f(w)\right\|\leq \sqrt{1+2\left\|f_n(w)\right\|^2}\sqrt{1+2\left\|f(w)\right\|^2}\chi_e(f_n(w),f(w))$$
for all $n\geq n_0$ and for all $w\in E.$ But $f$ is bounded on $E$ and $\{f_n\}$ is bounded on $E$ for all $n\geq n_0$, so it follows from the last inequality that $\{f_n\}$ converges uniformly to $f$ on $E.$ \end{proof}

\medskip

The notion of continuity with respect to the bicomplex chordal metric is given in the following definition.
\begin{defin}
A function $f$ is \textbf{bispherically continuous} at a point $w_0\in \mathbb{T}$ if,
given $\epsilon>0$, there exists $\delta>0$ such that
$$\chi_e(f(w),f(w_0))<\epsilon,$$
whenever $\left\|w-w_0\right\|<\delta$.
\end{defin}

In the case of \textbf{bicomplex meromorphic functions} we have the following result.

\begin{thm}
If $f(w)$ is a bicomplex meromorphic function in a domain $E\subset \mathbb{T}$, then $f$ is
bispherically continuous in $E$.
\end{thm}
\begin{proof}
Since $f(w)$ is a bicomplex meromorphic function on $E$,
then there exist meromorphic functions (see Thm. \ref{mero1})
$f_{e1}:E_1\longrightarrow\mathbb{C}(\bold{i_1})$ and
$f_{e2}:E_2\longrightarrow\mathbb{C}(\bold{i_1})$ with
$E_1=\mathcal{P}_1(E)$ and $E_2=\mathcal{P}_2(E)$ such that
$$f(z_1+z_2\bold{i_2})=f_{e1}(z_1-z_2\bold{i_1})\bold{e_1}+f_{e2}(z_1+z_2\bold{i_1})\bold{e_2} \mbox{ }\forall
\mbox{ }z_1+z_2\bold{i_2}\in E.$$
If $f$ is $\mathbb{T}$-holomorphic at $w_0\in E$, then $f_{ei}$ is holomorphic on $E_i$ for $i=1,2$. Hence, it is bispherically continuous on $E$ since
\begin{equation}
\chi_e(f(w),f(w_0))\leq\left\|f(w)-f(w_0)\right\|.
\label{meroza}
\end{equation}
If $w_0$ is a strong pole, then $\frac{1}{f_{e1}}$ and $\frac{1}{f_{e2}}$ is continuous at $\mathcal{P}_1(w_0)$ and
$\mathcal{P}_2(w_0)$ respectively. Moreover, noting that
\begin{equation*}
\mathbb \chi_e(f(w),f(w_0)) = \mathbb \chi_e(\frac{1}{f(w)},\frac{1}{f(w_0)})
\end{equation*}
\begin{equation*}
= \sqrt{\displaystyle\frac{\chi^{2}\left(\displaystyle\frac{1}{f_{e1}(\mathcal{P}_1(w))},\displaystyle\frac{1}{f_{e1}(\mathcal{P}_1(w_0))}\right) + \chi^{2}\left(\displaystyle\frac{1}{f_{e2}(\mathcal{P}_2(w))},\displaystyle\frac{1}{f_{e2}(\mathcal{P}_2(w_0))}\right)}{2}},
\end{equation*}
the result follows as in the preceding case.
If $w_0$ is a weak pole, then $\frac{1}{f_{e1}}$ or $\frac{1}{f_{e2}}$ is continuous at  $\mathcal{P}_1(w_0)$ or
$\mathcal{P}_2(w_0)$ respectively. Suppose, without loss of generality, that $\frac{1}{f_{e1}}$ is continuous at $\mathcal{P}_1(w_0)$ with $f_{e2}$ continuous at $\mathcal{P}_2(w_0)$.
Then,
$\mathbb \chi_e(f(w),f(w_0))=$
\begin{equation*} \sqrt{\displaystyle\frac{\chi^{2}\left(\displaystyle\frac{1}{f_{e1}(\mathcal{P}_1(w))},\displaystyle\frac{1}{f_{e1}(\mathcal{P}_1(w_0))}\right) + \chi^{2}(f_{e2}(\mathcal{P}_2(w),f_{e2}(\mathcal{P}_2(w_0)))}{2}},
\end{equation*}
and the result follows using the Equation \eqref{meroza} in the complex plane (in $\bold{i_1}$). \end{proof}

\begin{defin}A family $\boldsymbol F$ of bicomplex functions defined on a domain $\Omega\subset\mathbb{T}$ is said to be \textbf{bispherically equicontinuous}
at a point $w_{0}\in\Omega$ if for each $\epsilon > 0, \ \exists\delta = \delta(\epsilon,w_{0})$
such that $$\mathbb \chi_e \left(f(w),f(w_{0})\right)<  \epsilon \ \ \text{whenever}\ \left\|w-w_{0}\right\|< \delta  \,  \ \forall f \in \boldsymbol F. $$
Moreover, $\boldsymbol F$ is bispherically equicontinuous on a subset $E\subset\Omega$ if it is bispherically equicontinuous
at each point of $E$.
\end{defin}

\begin{rem} Since $$\mathbb \chi_e\left(f(w),f(w_{0})\right) \leq \left\|f(w)-f(w_{0})\right\|,$$
we see that equicontinuity with respect of the euclidean metric implies bispherical equicontinuity.
\end{rem}

\section{Normal Families of Bicomplex Meromorphic Functions}
\subsection{Basic results}

\begin{defin}
A family $\boldsymbol F$ of bicomplex meromorphic functions in a domain $\Omega\subset\mathbb{T}$ is \textbf{normal} in $\Omega$
if every sequence $\{f_{n}\}\subset\boldsymbol F$ contains a subsequence which converges bispherically uniformly on compact subsets of $\Omega$.
\end{defin}

That the limit function is either bicomplex meromorphic in $\Omega$ or in the set
$I^{-}_{\infty}$ or identically $\infty$ is a consequence of Corollary \ref{meroinft}. That the limit
can actually be identically $\infty$ is given by the following example.

\begin{exa}
Let $f_n(w)=\frac{n}{w}$, $n=1,2,3,\ldots,$ on the Lie Ball $D(0;r,r)$. Then each $f_n$ is bicomplex meromorphic and
$\{f_n\}$ converges bispherically uniformly to $\infty$ in $D(0;r,r)$.
\end{exa}

\begin{thm}
A family $\boldsymbol F$ of bicomplex meromorphic functions is normal in a domain $\Omega$ with respect to the bicomplex chordal metric if and only if
the family of meromorphic functions $F_{ei}=\mathcal{P}_i(\boldsymbol F)$ is normal in $\mathcal{P}_i(\Omega)$ for $i=1,2$ with respect to the chordal metric.
\label{promero}
\end{thm}
\begin{proof}
Suppose that $\boldsymbol F$ is normal in $\Omega$ with respect to the bicomplex chordal metric.
Let $\{(f_n)_1\}$ be a sequence in $\boldsymbol F_{e1}=\mathcal{P}_{1}(\boldsymbol F)$. We want to prove, without loss of generality, that
the family of meromorphic functions $\{(f_n)_1\}$ contains a subsequence which converges spherically locally uniformly on $\mathcal{P}_1(\Omega)$.
By definition, we can find a sequence $\{f_n\}$ in $\boldsymbol F$ such that $\{\mathcal{P}_1(f_n)\}=\{(f_n)_1\}$. Moreover, for any
$z_0\in\mathcal{P}_1(\Omega)$, we can find a $w_0\in\Omega$ such that $\mathcal{P}_1(w_0)=z_0$. Now, consider a closed $\mathbb{T}$-disk
$\overline{D}(w_0;r,r)$ in $\Omega$. By hypotheses, the sequence $\{f_{n}\}$ contains a subsequence $\{f_{n_k}\}$ which converges
bispherically uniformly on $\overline{D}(w_0;r,r)$. Hence, $\mathcal{P}_1(f_{n_k})=(f_{n_{k}})_1$ converges spherically uniformly on
$\overline{B^1(z_0,r)}\subset\mathcal{P}_1(\Omega)$.

Conversely, suppose that $\boldsymbol F_{ei}=\mathcal{P}_{i}(\boldsymbol F)$ is normal on $\mathcal{P}_{i}(\Omega)=\Omega_i$ for $i=1,2.$
We want to show that $\boldsymbol F$ is normal in $\Omega$ with respect to the bicomplex chordal metric. Let $\{f_n\}$ be any sequence in $\boldsymbol F$ and
$K$ be any compact subset of $\Omega$. Then $\{\mathcal{P}_1(f_n)\}=\{(f_n)_1\}$ is a sequence in
$\boldsymbol F_{e1}=\mathcal{P}_{1}(\boldsymbol F).$ Since $\boldsymbol F_{e1}=\mathcal{P}_{1}(\boldsymbol F)$ is normal in
$\mathcal{P}_1(\Omega)$ then $\{(f_n)_1\}$ has a subsequence $\{(f_{n_k})_1 \}$ which converges spherically
uniformly on $\mathcal{P}_1(K)$ to a $\overline{\mathbb{C}(\bo)}$-function. Now, consider $\{f_{n_k}\}$ in $\boldsymbol F$.
Then $\{\mathcal{P}_2(f_{n_k})\}=\{(f_{n_k})_2\}$ is a sequence in
$\boldsymbol F_{e2}=\mathcal{P}_{2}(\boldsymbol F)$. Since $\boldsymbol F_{e2}=\mathcal{P}_{2}(\boldsymbol F)$ is normal in
$\mathcal{P}_1(\Omega)$ then $\{(f_{n_k})_2\}$ has a subsequence $\{(f_{n_{k_l}})_2 \}$ which converges spherically
uniformly on $\mathcal{P}_2(K)$ to a $\overline{\mathbb{C}(\bo)}$-function. This implies that $\{(f_{n_{k_l}})_1\eo+(f_{n_{k_l}})_2\et\}$
is a subsequence of $\{f_n\}$ which converges bispherically uniformly on $\mathcal{P}_1(K)\times_{e}\mathcal{P}_2(K)\supseteq K$ to
a $\overline{\mathbb{T}}$-function showing that $\boldsymbol F$ is normal in $\Omega$ with respect to the bicomplex chordal metric. \end{proof}

\medskip
Since the limit function of a locally convergent sequence of meromorphic functions is either meromorphic or identically equal to $\infty$,
we have automatically the following result as a direct consequence of Theorems \ref{mero1} and \ref{promero}.

\begin{cor}
Let $\{f_{n}\}$ be a sequence of bicomplex meromorphic functions on $\Omega$ which converges bispherically uniformly on compact subsets to $f$.
Then $f$ is either a bicomplex meromorphic function on $\Omega$ or in the set
$I^{-}_{\infty}$ or identically $\infty$.
\label{meroinft}
\end{cor}

Moreover, from the fact that a family of analytic functions is
normal with respect to the usual metric if and only if the family
is normal with respect to the chordal metric (see \cite{Schiff}, Cor.
3.1.7)  and from the characterization of the notion of normality
for a family of bicomplex holomorphic functions (see \cite{Sharma}, Thm.
8), we obtain the following result as a consequence of Theorems \ref{theo5} and \ref{promero}.

\begin{cor}
A family $\boldsymbol F$ of $\mathbb{T}$-holomorphic functions is normal in a domain $\Omega$ with respect to the euclidian metric
if and only if $\boldsymbol F$ is normal in $\Omega$ with respect to the bicomplex chordal metric.
\end{cor}

\subsection{Bicomplex Montel Theorem}
In this subsection, we will give a proof of a bicomplex version of the Montel theorem for a family of bicomplex meromorphic functions.
We start with the following results.

\begin{lem}
If $\{f_{n}\}$ is the sequence of bispherically continuous functions which converges bispherically uniformly to a function
$f$ on a compact subset $E\subset\mathbb{T}$. Then $f$ is uniformly bispherically continuous on $E$
and the functions $\{f_{n}\}$ are bispherically equicontinuous on $E$.
\label{biequi}
\end{lem}
\begin{proof} The proof is same, with necessary changes, as that of
one complex variable analogue (see \cite{Schiff}, Prop. 1.6.2).\end{proof}

\begin{lem}
The bicomplex Riemann sphere is a compact metric space.
\label{biequi2}
\end{lem}
\begin{proof}
We will prove that $\overline{\mathbb{T}}$ is sequentially compact. Let $\{w_n\}$ be a sequence in $\overline{\mathbb{T}}$.
We have that $\{\mathcal{P}_i(w_n)\}$ is a sequence in $\overline{\mC(\bo)}$ for $i=1,2$. Since the Riemann sphere is
the one-point compactification of the complex plane, $\{\mathcal{P}_1(w_n)\}$ has a spherically convergent subsequence $\{\mathcal{P}_1(w_{n_k})\}$ in $\overline{\mC(\bo)}$
and $\{\mathcal{P}_2(w_{n_k})\}$ has also a spherically convergent subsequence $\{\mathcal{P}_2(w_{{n_k}_l})\}$ in $\overline{\mC(\bo)}$ such that
$\{\mathcal{P}_i(w_{{n_k}_l})\}$ converges spherically in $\overline{\mC(\bo)}$ for $i=1,2$.
Hence, $\{w_{{n_k}_l}\}$ converges bispherically in $\overline{\mathbb{T}}$.\end{proof}

\medskip
As for one complex variable, in discussing the normality of a family of bicomplex meromorphic functions, the concept of local boundedness is not entirely relevant.
However, bispherical equicontinuity can be substituted in the following counterpart of Montel's theorem.

\begin{thm}
A family $\boldsymbol F$ of bicomplex meromorphic functions in a bicomplex
domain $\Omega\subset\mathbb{T}$ is normal if and only if $\boldsymbol F$ is bispherically equicontinuous in $\Omega$.
\end{thm}
\begin{proof}
Suppose $\boldsymbol F$ is normal but not bispherically equicontinuous in $\Omega$. Then there is a point $w_{0} \in \Omega$,
some $\epsilon > 0 $, a sequence $\{w_{n}\}\longrightarrow w_{0}$ and a sequence $\{f_{n}\} \subset \boldsymbol F$ such that
\begin{equation}
\mathbb \chi_e\left(f_n(w_{0}),f_n(w_{n})\right) > \epsilon, \ \ n = 1,2,3\ldots.
\label{eps}
\end{equation}
Since $\boldsymbol F$ is normal, so $\{f_{n}\}$ has a subsequence $\{f_{n_{k}}\}$ converging bisherically uniformly on compact subsets of $\Omega$ and in particular on a compact subset containing $\{w_{n}\}$.
By the Lemma \ref{biequi}, this implies that $\{f_{n_{k}}\}$ is  bispherically equicontinuous at $w_{0}$.
This is a contradiction with the Equation \eqref{eps}. Therefore $\boldsymbol F$ is bispherically equicontinuous.

Conversely, let $\boldsymbol F$ be a bispherically equicontinuous family of bicomplex meromorphic functions defined on $\Omega$.
To show that $\boldsymbol F$ is normal in $\Omega$ we need to extract a locally bispherically uniformly convergent subsequence from every sequence in $\boldsymbol F.$
Let $\{f_{n}\}$  be any sequence in $\boldsymbol F$ and let $E$ be a countable dense subset of $\Omega$, for example we can take $E\bigcap \Omega$ where
$E=\{w_{n}=w_{1,n}\bold{e_1}+w_{2,n}\bold{e_2}: w_{j,n}=x_{j,n}+i_{1}y_{j,n}\text{ where } x_{j,n}, y_{j,n} \in \mathbb Q, j=1,2 \}.$
Take any sequence $\{f_{n}\}\subset \boldsymbol F$ and consider the sequence of bicomplex numbers $\{f_{n}(w_1)\}$. Since the bicomplex Riemann sphere is a \textbf{compact metric space}
(see Lemma \ref{biequi2}), $\{f_{n}(w_{1})\}$
has a subsequence $\{f_{n,1}\}$ converging bispherically at $w_1$. Next, consider the sequence
$\{f_{n,1}(w_{2})\}$, we can also find a subsequence $\{f_{n,2}\}$ of $\{f_{n,1}\}$ such that $\{f_{n,2}(w_{2})\}$ converges bispherically at $w_2$. Since $\{f_{n,2}\}$
is a subsequence of $\{f_{n,1}\},$ $\{f_{n,2}(w_{1})\}$ also converges bispherically at $w_1$. Therefore, $\{f_{n,2}\}$  converges bispherically at $w_1$ and $w_2$. Continuing this process,
for each $k\geq 1$ we obtain a subsequence $\{f_{n,k}\}$ that converges bispherically at $w_1, w_2, \ldots, w_k$ and $\{f_{n,k}\} \subset \{f_{n,k-1}\}$. Now by
Cantor's diagonal process we define a sequence $\{g_n\}$ as
$$g_{n}(w)=f_{n,n}(w), \ \ \ n\in \mathbb N.$$
Hence, $\{g_n(w_k)\}$ is a subsequence of the bispherically convergent sequence
$\{f_{n,k}(w_k) \}_{n\geq k}$ and hence converges for each $w_k \in E$.
Now, by hypothesis, $\{g_n \}$ is bispherically equicontinuous on every compact subset of $\Omega$.
So for every $\epsilon>0$ and every compact subset $K$ of $\Omega$ there is a $\delta>0$ such that
\begin{equation}
\chi_e (g_n(w),g_n(w^{\prime}))< \frac{\epsilon}{3},\ \\\ \forall n\in\mathbb N  \text{ and } \forall \ w, w^{\prime} \in K \text{ with } \left\|w-w^{\prime} \right\|< \delta.
\label{eq02}
\end{equation}
Since $K$ is compact, we can cover it by a finite subcover, say $$K\subset\bigcup^p_{j=1}\{B^2(\varsigma_{j},\frac {\delta}{2}): \varsigma_{j} \in E\}.$$
Since $\varsigma_{j}\in E$,  $\{g_{n}(\varsigma _{j})\}$ converges for each $j:1\leq j \leq p$ which further implies that
$\{g_{n}(\varsigma _{j})\}$ is a Cauchy sequence. That is, there is a positive integer $n_0$ such that
\begin{equation}
\chi_e(g_{n}(\varsigma _{j}),g_{m}(\varsigma _{j})) <\frac{\epsilon}{3},\mbox{ } \forall \ m, n \geq n_{0}, \ (1\leq j \leq p).
\label{eq03}
\end{equation}
Finally, for any $w\in K$, $w\in B^2(\varsigma_{j},\frac {\delta}{2})$ for some $1\leq j_0 \leq p$. Thus, from Equations \eqref{eq02} and \eqref{eq03}, we have
$$ \chi_e(g_{n}(w), g_{m}(w)) \leq \chi_e(g_{n}(w), g_{n}(\varsigma _{j_0}))+\chi_e (g_{n}(\varsigma _{j_0}), g_{m}(\varsigma _{j_0}))+\chi_e(g_{n}(\varsigma _{j_0}), g_{m}(w))$$
$$< \frac{\epsilon}{3} + \frac{\epsilon}{3} + \frac{\epsilon}{3}, \\\ \forall m,n\geq n_0.$$
Therefore, by construction, $\ \{g_{n}\}$ is locally bispherically uniformly Cauchy and hence converges locally bispherically uniformly on $\Omega$.\end{proof}

\subsection{Fundamental Normality Test}
Finally, we prove the bicomplex version of the Fundamental Normality Test for meromorphic functions. First,
we prove it for bicomplex holomorphic functions.
\begin{thm} Let $\boldsymbol F$ be a family of bicomplex holomorphic functions in a domain $\Omega\subset \mathbb {T}$.
Suppose there are $\alpha, \ \beta \in \mathbb{T}$ such that\\
(a) $\alpha - \beta$ is invertible, and \\
(b) $S\cap \mathcal{R}(f) =\varnothing, \ \forall f\in \boldsymbol F,$ where $S=\{w\in \mathbb{T}:w-\alpha \in \mathcal{NC}\}\cup
\{w\in \mathbb{T}:w-\beta \in \mathcal{NC}\}$ and $\mathcal{R}(f)$ denotes the range of $f.$\\
 Then $\boldsymbol F$ is a normal family in $\Omega.$
\label{FNTH}
\end{thm}
\begin{proof} Conditions $(a)$ and $(b)$ of the hypothesis imply that for each $f\in \boldsymbol F$ the projection
$P_i(f)$ does not assume $P_i(\alpha)$ and $P_i(\beta)$, where $P_i(\alpha)\neq P_i(\beta)$, for $i=1,2$. Then by the fundamental normality test for holomorphic functions (see \cite{Schiff}), it follows that $P_i(\boldsymbol F)$ is normal in $P_i(\Omega)$ for $i=1,2.$ Now by Theorem 11 of \cite{Sharma} we conclude that $\boldsymbol F$ is normal in $\Omega.$
\end{proof}

\smallskip
Following \cite{Rochon2}, one can easily obtain a bicomplex version of the Picard's Little Theorem for meromorphic functions.
\begin{thm} Let $f$ be a bicomplex meromorphic function in $\mathbb {T}$. Suppose there exist $\alpha, \ \beta, \ \gamma \in \mathbb {T}$ such that\\
(a) $\alpha - \beta, \ \beta - \gamma, \ \gamma -\alpha$ are invertible, and\\
(b)  $S\cap \mathcal{R}(f) =\varnothing, \ \forall f\in \boldsymbol F,$ where $S=\{w\in \mathbb{T}:w-\alpha \in \mathcal{NC}\}\cup \{w\in \mathbb{T}:w-\beta \in \mathcal{NC}\}\cup \{w\in \mathbb{T}:w-\gamma \in \mathcal{NC}\}$ and $\mathcal{R}(f)$ denotes the range of $f.$\\
  Then $f$ is a constant function.
\end{thm}
\begin{thm} Let $\boldsymbol F$ be a family of bicomplex meromorphic functions defined in a domain $\Omega \subset \mathbb{T}.$ Suppose there exist $\alpha, \ \beta, \ \gamma \in \mathbb {T}$ such that\\
(a) $\alpha - \beta, \ \beta - \gamma, \ \gamma -\alpha$ are invertible, and\\
(b)  $S\cap \mathcal{R}(f) =\varnothing, \ \forall f\in \boldsymbol F,$ where $S=\{w\in \mathbb{T}:w-\alpha \in \mathcal{NC}\}\cup \{w\in \mathbb{T}:w-\beta \in \mathcal{NC}\}\cup \{w\in \mathbb{T}:w-\gamma \in \mathcal{NC}\}$ and $\mathcal{R}(f)$ denotes the range of $f.$\\
Then $\boldsymbol F$ is normal in $\Omega.$
\end{thm}
\begin{proof} Following the method of proof of Theorem \ref{FNTH} and applying Theorem \ref{promero} and
the fundamental normality test for meromorphic functions (\cite{Schiff}, Page 74) we can easily conclude that the
family $\boldsymbol F$ is normal in $\Omega$.
\end{proof}

\subsection*{Acknowledgements}
DR is grateful to the Natural Sciences and Engineering Research Council of Canada for financial support.

\end{document}